\documentclass[a4paper,12pt]{article}

\usepackage[top=2.5cm, bottom=2.5cm, left=2.5cm, right=2.5cm]{geometry}
 \usepackage{indentfirst}
\usepackage{graphicx}
\usepackage{enumerate}
\usepackage{amsthm}
\usepackage{amsfonts}
\usepackage{amsthm, amssymb, amsmath}

\usepackage{CJK}

\newtheorem{thm}{Theorem}[section]
\newtheorem{rem}[thm]{Remark}

\newtheorem{lem}[thm]{Lemma}

\numberwithin{equation}{section}

\title{\textbf{Generalized Pontryagin Maximum Principle in port-Hamiltonian Systems and Application} }
\author{{ Yuanzun Zhao} \\
{Department of Mathematics}\\ {Peking University,
Beijing 100871, China}\\
{\sf email: lukesweet@163.com} }
\date{}

\begin{document}

\maketitle
\renewcommand{\abstractname}{}
\begin {abstract}
\noindent
  {\bf Abstract}
{ }Based on Pontryagin Maximum Principle (PMP), this paper establishes a generalized PMP aiming at control system with with extra input/output terms. The paper details the adaptive target and gives a proof of the generalized theorem. Transformation and application of this specific method is showed and thereafter an example explained its feasibility.

\end{abstract}
\newpage

\section{{\textbf{Introduction}}}
Optimal control, a crucial component of modern control theory, has already been generalized to many kinds of control systems. One of the most used method to deal with optimal control problems is Pontryagin Maximum Principle(PMP)\cite{qy1}. The core idea of PMP is to embed the target control system into a Hamiltonian system and transform the optimal requirement into a boundary condition\cite{qy2}. However, not all the problems can be solved through PMP.

Towards control problems with extra terms in the control equation or extra algebraic expression, which cannot be easily solved by classic PMP, it's effective to embed the systems into port-Hamiltonian. The embedding method, analogous to handling of PMP\cite{An1}, is called Pontryagin Maximum Principle for Port-Hamiltonian Systems and will be detailed in the following. All the conclusions is based on the port-Hamiltonian systems theory and is proved of its reliability and maturity.\cite{pa1}

\vspace{2mm}
 \noindent {\bf Acknowledgement:} I would like to thank Mr. Zhangju Liu for his useful advice and great encouragement.
\vspace{1cm}

\section{{\textbf{port-Pontryagin Maximum Principle and Application}}}
\subsection{{\textbf{Generalized Optimal Control Problems in port-Hamiltonian Systems}}}
Consider the special control system described as follows:
\begin{equation}
\dot{q}=F_u (q)+B(q)f',
\end{equation}
\begin{equation}
e=A^{T}(q)(\dot q-B(q)f')
\end{equation}
$q\in M,    u=u(t)\in U$, $f'=f'(t),e=e(t)$ is a continuous function in $N$,
with initial conditions
\begin{equation}
q(0)=q_0,
\end{equation}
Given cost function $¦Õ:M\times U\times T^{*}N \rightarrow  \mathbb{R}$, the problem is to find a pair of control $u=\tilde{u}$ and input that minimizes $$J=\int_0^{t_1} \varphi(q(t),u(t),e(t)) dt.$$

This is a typical control problem in port-Hamiltonian system, with $B(q)f$ as a port of any kind. Notice that $f$ is a continuous input(different from $u$), that is, it cannot be simply viewed as a part of control. The basic idea dealing with control problems of this kind is dividing it into two parts: finding out the optimal $u$ for fixed $f$, and then finding out the optimal continuous $f$.

To solve the former part, we need to establish a generalized PMP aimed at the port. It's required to embed the control system in a port-Hamiltonian system, just like PMP embedding control systems in Hamiltonian systems. While the latter part requires transformation and application of PMP.

The two parts are respectively discussed as follows.
\subsection{{\textbf{Generalization to PMP in port-Hamiltonian system }}}
We now consider embedding the optimal control problem in the following port-Hamiltonian system:
\begin{equation}
\left\{
\begin{aligned}
&\frac{\partial q}{\partial t}=\frac{\partial H_u}{\partial p}+B(q)f'\\
&\frac{\partial p}{\partial t}=-\frac{\partial H_u}{\partial q}+A(q)f\\
&e'=B^{T}(q)\frac{\partial H_u}{\partial q}\\
&e=A^{T}(q)\frac{\partial H_u}{\partial p}
\end{aligned}
\right.
\end{equation}
whereas $M\subset \mathbb{R}^n$ is a configuration space, $U\subset \mathbb{R}^l$ is admissible control set, $N\subset \mathbb{R}^k$ is admissible input set. $q\in M$ is the state of system, and $u(t)\in U$ is a preinstalled control, $A(q), B(q)$ is $n\times k$ transform matrixes depending on state $q$.

With Hamiltonian defined as $h_{u,f}(\lambda)=\langle \lambda,f_u (\lambda) \rangle, \lambda\in T^{*}M$, thereupon we have the following geometric statement of port-PMP:
\theoremstyle{plain} \newtheorem{theorem}{Theorem}[section]
\begin{thm}[PMP for port-Hamiltonian systems]Given an admissible control $u=\tilde{u}(t),f=\tilde{f}(t)$, if (2.2),(2.3) have a solution $\tilde{q}(t)=q_{\tilde{u}(t),\tilde{f}(t)}(t)$, if\\
\begin{equation*}
\tilde{q}(t_1)\in \partial A_{q_0}(t_1),
\end{equation*}
here $A_{q_0}(t_1)$ indicates the attainable set from point $q_0$ in $t_1$ time, whatever $u(t)$,\\
with $f(t)=\tilde f(t)$ fixed.\\
Then there exists a Lipschitzian curve $\lambda_t\in T_{\tilde{q}(t)}^* M,t\in [0,t_1]$ in the cotangent bundle, such that\\
\begin{equation}
\lambda_t\not\equiv0,
\end{equation}
\begin{equation}
\frac{\partial \lambda}{\partial t}=-\frac{\partial h_u}{\partial q}+A(q)f,
\end{equation}
\begin{equation*}
h_{\tilde{u}(t),\tilde{f}(t)}(\lambda_t)+\int_0^{t} e(\tilde{u},\tau)f(\tau)+e'(\tilde{u},\tau)f'(\tau)d\tau
\end{equation*}
\begin{equation}
=\max_{u\in U}\{h_{u(t),\tilde{f}(t)}(\lambda_t)+\int_0^{t} e(u,\tau)f(\tau)+e'(u,\tau)f'(\tau)d\tau\}
\end{equation}
for almost all $t\in [0,t_1]$
\end{thm}

\begin{proof}
To prove the conclusion, a vector field depending on two parameters needs to introduced first:\\
\begin{equation*}
g_{\tau,u}={P_\tau^{t_1}}_*(F_{u(\tau)}-F_{\tilde u(\tau)}), \tau \in [0,t_1],u\in U
\end{equation*}

According to \cite{An1},
\begin{equation*}
q_u(t_1)=q_1\circ \vec{exp} \int_0^{t_1} g_\tau,u(\tau) dx
\end{equation*}

Thus we have the following lemmas:
\begin{lem}\label{Lem1}Let $\Gamma \in[0,t_1]$ be the set of Lebesgue points of the control $\tilde{u} (\cdot)$. If
\begin{equation*}
T_{q_1}M=cone\{g_{\tau,u} (q_1)|\tau \in T ,u\in U\}
\end{equation*}
Then
\begin{equation*}
q_1\in intA_{q_0}(t_1)=A_{q_0}(t_1)\backslash \partial A_{q_0}(t_1)
\end{equation*}
\end{lem}
The proof is in \cite{An1}.

Now in the port-Hamiltonian System, we define a new flow $g^{P}$
\begin{equation*}
g_{\tau,u,f}={P_\tau^{t_1}}_*((F_{u(\tau)}+B(q(t))f'(t))-(F_{\tilde u(\tau)}+B(q(t))f'(t)), \tau \in [0,t_1],u\in U
\end{equation*}
\begin{lem}\label{Lem2}Let $\Gamma \in[0,t_1]$ be the set of Lebesgue points of the control and input $\tilde{u} (\cdot),\tilde{f} (\cdot)$. If
\begin{equation*}
T_{q_1}M=cone\{g_{\tau,u,f} (q_1)|\tau \in T ,u\in U\, f\in Q\}
\end{equation*}
Then
\begin{equation*}
q_1\in intA_{q_0}(t_1)=A_{q_0}(t_1)\backslash \partial A_{q_0}(t_1)
\end{equation*}
\end{lem}
\begin{proof}
Define $\hat u=(u,f)$, $\hat u\in L^k\times C^{(1)l}\subset L^{k+l}$. Apply Lemma \ref{Lem1} to $\hat u$ and we directly get the conclusion.
\end{proof}

Now return to proof of the theorem.
Let the endpoint of trajectory $q_1$ satisfies $q_1=\tilde{q}(t_1)\in \partial A_{q_0}(t_1)$. By Lemma \ref{Lem2}, if this condition holds, origin point $o_{q_1}\in T_{q_1}M$ belongs to $\partial cone\{g_{\tau,u,f} (q_1)|\tau \in T ,u\in U, f\in N\}$, so this set has a hyperplane of support at the origin:
\begin{equation*}
\exists \lambda^{E*}_{t_1}\in T_{q_1}^*M, \lambda^{E*}_{t_1}\not=0,
\end{equation*}
satisfies:
\begin{equation*}
\langle \lambda^{E*}_{t_1},g_{t,u,f}(q_1)\rangle\leq 0, \forall t\in[0,t_1],u\in U
\end{equation*}
Let $\lambda_{t_1}^*=\lambda^{E*}_{t_1}+\int_0^{t_1} A(q(t))f(t)dt$

thus it satisfies:
\begin{equation*}
\langle \lambda_{t_1}^*+\int_0^{t_1} A(q(t))f(t)dt,g_{t,u,f}(q_1)\rangle\leq 0, \forall t\in[0,t_1],u\in U
\end{equation*}
that is
\begin{equation*}
\langle \lambda_{t_1}^{E*},P_{t*}^{t_1}(F_u(q_1)+\int_0^{t}B(q)f'd\tau)\rangle\leq \langle \lambda_{t_1}^{E*},P_{t*}^{t_1}(F_{\tilde{u}}(q_1)+\int_0^{t} B(q)f'd\tau)\rangle
\end{equation*}
further,
\begin{equation*}
\langle P_{t}^{t_1*}(\lambda_{t_1}^{E*}),F_u(q_1)+\int_0^{t}B(q)f'd\tau\rangle\leq \langle P_{t}^{t_1*}(\lambda_{t_1}^{E*}),F_{\tilde{u}}(q_1)+\int_0^{t}B(q)f'd\tau\rangle
\end{equation*}

Then flow $P_t^{t_1}$ defines Lipschitzian curve $\lambda_{t}^*$ in $ T_{q(t)}^*M$:
\begin{equation*}
\lambda_{t}^*\triangleq P_t^{t_1*}\lambda_{t_1}^*+\int_0^{t}A(q(\tau))f(\tau)d\tau\in T_{\tilde{q}(t)}^*M
\end{equation*}

In terms of this covector curve, the inequation above reads:
\begin{equation*}
\langle \lambda_{t}^*-\int_0^{t} A(q)fd\tau, F_u(q_1)+\int_0^{t}B(q)f'd\tau\rangle\leq \langle \lambda_{t}^*-\int_0^{t} A(q)fd\tau, F_{\tilde{u}}(q_1)+\int_0^{t}B(q)f'd\tau\rangle
\end{equation*}
that is,
\begin{equation*}
\langle \lambda_{t}^*,F_u(q_t)\rangle+\int_0^{t}e'(u,\tau)f'(\tau)+e(u,\tau)f(\tau)d\tau\leq \langle \lambda_{t}^*,F_{\tilde{u}}(q_t)\rangle+\int_0^{t}e'(\tilde{u},\tau)f'(\tau)+e(\tilde{u},\tau)f(\tau)d\tau
\end{equation*}
which exactly equals to Hamiltonian maximum conditions:
\begin{eqnarray*}
&&h_{\tilde{u}(t),\tilde{f}(t)}(\lambda_t)+\int_0^{t} e(\tilde{u},\tau)f(\tau)+e'(\tilde{u},\tau)f'(\tau)d\tau,\\
&&=\max_{u\in U}\{h_{u(t),\tilde{f}(t)}(\lambda_t)+\int_0^{t} e(u,\tau)f(\tau)+e'(u,\tau)f'(\tau)d\tau\}
\end{eqnarray*}

Besides, since the curve is fixed once the terminal point is fixed, the following equation holds:
\begin{equation*}
\lambda_{t}^*\triangleq P_t^{t_1*}\lambda_{t_1}^{E*}+\int_0^{t}A(q(\tau))f(\tau)d\tau=\lambda_{t_1}^*\circ (\vec{exp}\int_t^{t_1}-\frac{\partial h_{u,f}}{\partial q}+A(q(x))fdx)^*-\int_t^{t_1}A(q(\tau))f(\tau)d\tau
\end{equation*}
that is,
\begin{equation*}
\frac{\partial \lambda_t^*}{\partial t}=-\frac{\partial h_{u,f}}{\partial q}+A(q)f\\
\end{equation*}

Thus, the existence of extremal curve is proved.
\end{proof}
\vspace{1cm}

\section{{\textbf{Application of PMP for port-Hamiltonian System and Transformation of Control}}}
\subsection{{\textbf{Application of PMP for port-Hamiltonian System}}}
PMP for port-Hamiltonian System is adaptive to control systems with either extra input or output. Input $B(q)f'$ in (2.1) corresponds to storage port and output $e$ in (2.2) can be either dissipation port or another storage. Analogously to application of PMP to classic control optimal problems, we introduce an extremal parameter $\nu$.

Consider the following optimal control problem:
\begin{eqnarray}
&&\dot{q}=f_u(q), q\in M, u\in U\\
&&q(0)=q_0\in M\\
&&t_1 fixed
\end{eqnarray}
and the cost function is defined as
\begin{equation}
\hat{J}=\int_0^{t_1}\hat{\varphi}(q(t),u(t),e(t))dt,
\end{equation}

We extend this control system as follows:
\begin{equation*}
\hat{q}=\left(
\begin{array}{c}
J_{q_0}(u)\\
q
\end{array}
\right)
\end{equation*}
and extend the corresponding vector field:
\begin{equation*}
\hat{f}_u(q)=\left(
\begin{array}{c}
\varphi(q,u)\\
f_u(q)
\end{array}
\right)
\end{equation*}

Thus we get a new control system:
\begin{eqnarray}
&&\dot{\hat{q}}=\dot{f}_u(q), q\in M, u\in U\\
&&\dot{q}(0)=\dot{q}_0=\left(
\begin{array}{c}
0\\
q_0
\end{array}
\right)\\
&&t_1 fixed
\end{eqnarray}
where $\hat{q}_0$ is a (n+1)-dimension random vector.
\begin{rem}
Notice that if $\tilde{u}$ is optimal control, then the following condition holds:
\begin{equation*}
h_{\tilde{u}(t)}(\lambda_t^*)=\max_{u\in U} \{h_{u(t)}(\lambda_t^*)
\end{equation*}
Then the terminal point satisfies:
\begin{equation*}
\tilde{\hat{q}}(t_1)\in \partial A_{\hat{q}_0} (t_1),
\end{equation*}
Therefore, PMP for port-Hamiltonian Systems can be applied to find the extremal curve.
\end{rem}
\vspace{0.5cm}

\subsection{{\textbf{Introduction of extremal parameter}}}
When use PMP* to solve optimal control problems, we need transform it into Hamiltonian equations through the Hamiltonian. However, directly using the previous Hamiltonian does not extinguish maximum from minimum of cost function. To make up the defect, we introduce new parameters $\nu$ and a new control $w$.
Define $y=J_{q_0}(u)$, and consider the following system:
\begin{eqnarray}
&&\dot y=\varphi(q,u)+w\\
&&\dot q= f_u(q)
\end{eqnarray}

Then the extremal stochastic curve in origin system corresponding to the control $w(t)\equiv0$. As a result, it comes to the boundary of attainable set at $t_1$. Apply PMP to it.\\
Define new Hamiltonian as follows:
\begin{equation}
\hat{h}_{(w,u)}(\nu,\lambda^*)=\langle \lambda^*,f_u\rangle +\nu(\varphi+w)
\end{equation}

The corresponding Hamiltonian system is
\begin{equation}
\left\{
\begin{aligned}
&\frac{\partial \nu}{\partial t}=\frac{\partial \hat h}{\partial y}=0\\
&\frac{\partial y}{\partial t}=\varphi+w\\
&\dot\lambda_t^*=\vec h_{\tilde{u}(t)}(\lambda_t)
\end{aligned}
\right.
\end{equation}

The first equation stands for $v\equiv constant$.

In terms of this system, Hamiltonian Maximum condition is:
\begin{equation*}
(\langle \lambda_t^*,f_{\tilde{u}(t)}\rangle+\nu\varphi(\tilde{q}(t),\tilde{u}(t)))=\max_{u\in U,w\geq0} (\langle \lambda_t^*,f_{u(t)}\rangle+\nu\varphi(\tilde{q}(t),u(t)))+\nu w
\end{equation*}
Since the maximum of original system is attained, there must be $\nu\leq 0$, thus we can set $\nu=0$ in the right hand of maximum condition:
\begin{equation*}
(\langle \lambda_t^*,f_{\tilde{u}(t)}\rangle+\nu\varphi(\tilde{q}(t),\tilde{u}(t)))=\max_{u\in U,w\geq0} (\langle \lambda_t^*,f_{u(t)}\rangle+\nu\varphi(\tilde{q}(t),u(t)))
\end{equation*}
So we prove the following conclusion:
\begin{thm}If $u=\tilde u (t)$ is the optimal control for problem (3.9)-(3.11), that is, $\tilde u(t)$ minimizes $J(u)$. Define generalized Hamiltonian family: $h_{u,f}^\nu(\lambda)=\langle \lambda,f_u (\lambda) \rangle+\int_0^{t} e(\tau)f(\tau)+e'(\tau)f'(\tau)d\tau+\nu\varphi(q(t),u(t),e(t)), \lambda\in T^{*}M$£¬then there exists a stochastic Lipschitzian curve $\lambda_t\in T_{\tilde{q}(t)}^* M,t\in [0,t_1]$, and a number $\nu\in \mathbb{R}$ such that\\
\begin{eqnarray}
&&\lambda_t\not\equiv0,\\
&&\frac{\partial \lambda}{\partial t}=-\frac{\partial H_u}{\partial q}+A(q)f,\\
&&h\nu_{\tilde{u}(t),\tilde{f}(t)}(\lambda_t)=\max_{u\in U}\{h^\nu_{u(t),\tilde{f}(t)}(\lambda_t)\},\\
&&\nu\leq0,
\end{eqnarray}
for almost all $t\in [0,t_1]$
\end{thm}
\begin{rem}
since pair $(\lambda_t^*,\nu)$ can be multiplied by any positive number, only abnormal case $\nu=0$ and normal case $\nu=-1$ need to be considered. Besides, when solving maximum problems, $\nu\leq 0$ becomes $\nu\geq 0$, that is, analogously abnormal cases $\nu=0$ and normal case $\nu=1$.
\end{rem}

\vspace{0.5cm}

\section{{\textbf{Example}}}
To demonstrate the method more specifically, we consider its application to classic Cheapest Stop Problem.

The original problem can be described as follows: A train moves on the railway. We start braking the train at certain initial location and speed. The goal is to stop it with minimum expenditure of energy, which is assumed proportional to the integral of squared acceleration.

Its mathematic statements is
\begin{eqnarray*}
&&\ddot x=u, \\
&&x(0)=x_0\\
&&\dot x(0)=v_0\\
&&t_1 fixed, x(t_1)=x_1, \dot x(0)=0,
\end{eqnarray*}
find $u=\tilde u$ minimizing $J(u)=\int_0^{t_1}u^2dt$.

Now consider that an auxiliary stopping device is introduced. The device is a slope that can be slowly uplifted at the entrance of station, simultaneously saving energy in gravitational form. The saved energy can be partially reused. Our new goal is to stop the train with minimum expenditure of cost, that is, energy used in stopping plus brake abrasion, minus energy that can be reused. The reuse ratio is defined as $R_u$, while the lifting of slope is measured by a continuous variable $f$. As we can see, with $f$ going up, $\dot x$ and $x$ is affected. The braked abrasion is assumed linear to the integral of output $e$.

Its mathematic statement can be organized as follows:
\begin{eqnarray}
&&\left\{
\begin{array}{l}
\dot x_1=x_2\\
\dot x_2=u\\
\end{array}
\right.,x=\left(
\begin{array}{c}
x_1\\
x_2
\end{array}
\right)\in R^2, u\in R\\
&&x(0)=x^{(0)}\\
&&t_1 fixed,\\
&&\dot x=\left(\begin{array}{c}
x_2\\
u
\end{array}\right)+B(x)f\\
&&e_1=A(x)\dot x\\
&&e_2=B(x)\dot x\\
&&\hat J(u)=\int_0^{t_1}u^2+(e_1+e_2)fdt\rightarrow min.
\end{eqnarray}
The system has two obvious output $e_1$ and $e_2$, thus be applied with PMP for port-Hamiltonian systems.
Its generalized Hamiltonian is
\begin{equation*}
h_{u,f,t}^\nu(\xi^*,x)=\xi^*_1 x_2+\xi^*_2 u+\nu(u^2+(e_1+e_2)f)+\int_0^t(e_1+e_2)fd\tau, \xi^*=\left(
\begin{array}{c}
\xi_1^*\\
\xi_2^*
\end{array}
\right)\in T_{x(t)}^*R^2,
\end{equation*}
Thus it is embedded in a port-Hamiltonian systems:
\begin{equation}
\left\{
\begin{aligned}
&\frac{\partial q}{\partial t}=\frac{\partial H_u}{\partial \xi}+B(q)f\\
&\frac{\partial \xi}{\partial t}=-\frac{\partial H_u}{\partial q}+A(q)f\\
&e=A^{T}(q)\frac{\partial H_u}{\partial p}\\
&e'=B^{T}(q)\frac{\partial H_u}{\partial q}
\end{aligned}
\right.
\end{equation}
First consider abnormal case $\nu=0$:
\begin{equation*}
h_u^0(\xi^*,x)=\xi^*_1 x_2+\xi^*_2 u.
\end{equation*}
If maximum of $h_u^0(\xi^*,x)$ exists, there must be $E\xi^*_2\equiv0$. This means $\tilde u \equiv0$, which contradicts nontrivial requirement.\\
Then consider the normal case $\nu=-1$:
\begin{equation*}
h_u^{-1}(\xi^*,x)=\xi^*_1 x_2+\xi^*_2 u-(u^2+(e_1+e_2)f).
\end{equation*}
Thus,
\begin{equation*}
\left\{
\begin{aligned}
\dot\xi_1^*=\frac{\partial\xi_1^*}{\partial t}=\frac{\partial h_u^{-1}}{\partial x_1}=0\\
\dot\xi_2^*=\frac{\partial\xi_2^*}{\partial t}=\frac{\partial h_u^{-1}}{\partial x_2}=\xi_1^*
\end{aligned}
\right.
\end{equation*}
Due to the Hamiltonian maximum generated by U,
\begin{equation*}
h_u^{-1}(\xi^*,x)=\xi^*_1 x_2+\xi^*_2 u-(u^2+(e_1+e_2)f).
\end{equation*}
\begin{equation*}
\frac{\partial h_u^{-1}}{\partial u}=0
\end{equation*}
Thus we get
\begin{equation*}
\tilde u(t)=\frac{1}{2}\xi_2^*(t)=\alpha t+\beta.
\end{equation*}

The optimal control shall be linear. We can easily find the optimal control when put the conclusion into origin system. The control system with both control $u$ and input $f$ can be now transformed into a classic optimal problem with only continuous control $f$.

\vspace{1cm}

\section{{\textbf{Conclusion}}}
The paper is focused on certain kind of extended control problems and give a feasible solution to the problems with output or input terms. Core idea of the solution is to divide the control systems into two parts and find the optimal control for fixed input/output. The theoretic basis of this method is PMP for port-Hamiltonian Systems which is introduced in the second section. The method, in essence, is a generalization of classic PMP since it's obvious that the theorem regress to PMP when output/input become zero.

Up to now we can only make use of certain kind of port-Hamiltonian systems, specifically the ones with symplectic structure. It's yet to be researched how general port-Hamiltonian systems can be connected to control problems. Though port theory, as a new subject, has been perfectly developed during the years, there are still blanks in its connection with classic subjects. Beyond doubt, there is great potential in this subject. Port-Hamiltonian theory stand a good chance to make improvement in not only control theory, but other applied fields.

\end{document}